\documentclass[A4paper,10pt]{amsart}
\usepackage{amsmath,amsthm,amssymb,amsfonts,bbm}
\usepackage[all]{xy}
\usepackage{mathrsfs}
\usepackage{hyperref}

\newtheorem{thm}{Theorem}%[section]
\newtheorem{lem}[thm]{Lemma}
\newtheorem*{lem*}{Lemma}
\newtheorem*{conj*}{Conjecture}

\newtheorem{prop}[thm]{Proposition}
\newtheorem*{cor*}{Corollary}
\newtheorem*{prop*}{Proposition}

\newtheorem*{res}{Result}

\theoremstyle{definition}

\newtheorem*{def*}{Definition}
\newtheorem{rem}[thm]{Remark}
\newtheorem*{rem*}{Remark}

\DeclareMathOperator{\h}{H}

\newcommand{\triv}{{\mathbf{1}}}
\def\R{\mathbb R}
\def\N{\mathbb N}
\def\Z{\mathbb Z}
\def\A{\mathbb A}
\def\Q{\mathbb Q}
\def\C{\mathbb C}

\def\wt{\widetilde}
\def\ira{\stackrel{\sim}{\rightarrow}}

\def\ra{\rightarrow}

\def\g{\mathfrak g}

\def\n{\mathfrak n}
\def\a{\mathfrak a}
\def\b{\mathfrak b}
\def\q{\mathfrak q}

\def\ep{\varepsilon}
\def\h{\mathfrak h}
\def\l{\mathfrak l}
\def\m{\mathfrak m}

\def\p{\mathfrak p}

\def\wij{w_{(I,J)}}
\def\<{\langle}
\def\>{\rangle}
\def\i{\mathbbm i}
\def\j{\mathbbm j}

\title[]%
{On the Eisenstein Cohomology of Odd Orthogonal Groups}
\author{Gerald Gotsbacher \and Harald Grobner}
\date{\today}
\address{Gerald Gotsbacher: Department of Mathematics, University of Toronto, 40 St. George St., Toronto, ON, M5S 2E4 Canada}
\email{gerald@math.tifr.res.in}

\address{Harald Grobner: {\it Current address}: Max-Planck Institut f\"ur Mathematik\\ Vivatsgasse 7\\ 53111 Bonn, Germany}
\email{harald.grobner@univie.ac.at}
\urladdr{http://homepage.univie.ac.at/harald.grobner}

\thanks{The second named author was supported financially by the Max-Planck Institut f\"ur Mathematik (MPIM), Bonn, the Erwin Schr\"odinger Institute for Mathematical Physics (ESI), Vienna, and the Austrian Science Found (FWF), project no. P21090.}
\thanks{This is the preprint of an article which is going to appear in {\it Forum Math.}}

\begin{document}

\begin{abstract}
The paper investigates a significant part of the automorphic, in
fact of the so-called Eisenstein cohomology of split odd orthogonal
groups over $\Q$. The main result provides a description of residual
and regular Eisenstein cohomology classes for \emph{maximal}
parabolic $\Q$-subgroups in case of \emph{generic} cohomological
cuspidal automorphic representations of their Levi subgroups. That
is, such identifying necessary conditions on these latter
representations as well as on the complex parameters in order for
the associated Eisenstein series to possibly yield non-trivial
classes in the automorphic cohomology.
\end{abstract}
\maketitle

\section*{Introduction}
The main objective of the effort to be unfolded here is to study the
Eisenstein cohomology of the $\Q$-split odd orthogonal groups
$G=SO_{2n+1}$. Specifically, it is the contribution to the latter
stemming from maximal parabolic $\Q$-subgroups that is dealt with.

To put ourselves in medias res let $\g$ be the Lie algebra of the
group of real points of $G$, $K$ a maximal compact subgroup therein
and $\mathcal{A}(G)$ the $(\g,K)$-module of ad\`elic automorphic
forms on $G$. If $E$ is a finite dimensional irreducible rational
representation of $G(\R)$, the automorphic cohomology of $G$ twisted
by $E$ is defined to be $H^*(\g,K;\mathcal{A}(G)\otimes E)$. Let
then $P\subset G$ be a maximal parabolic $\Q$-subgroup with Levi subgroup $L$ and maximal central
$\Q$-split torus $A\subseteq L$. We consider the
Eisenstein series $E_P(f,\Lambda)$ associated to a cohomological,
globally generic, cuspidal automorphic representation $\pi$ of
$L(\A)$, resp. an element $f$ in the representation induced from
$\pi$ and a complex parameter $\Lambda\in\a^*_\C$.
Langlands' theory of Eisenstein series ensures convergence of
$E_P(f,\Lambda)$ on a right half plane with respect to $\Lambda$,
the existence of a meromorphic continuation to the entire complex
plane $\a^*_{\C}$ and gives a restriction on the location of its
possible poles in terms of affine lines in it. Moreover,
$E_P(f,\Lambda)$ (resp. its residue) gives rise to an element of
$\mathcal{A}(G)$ for fixed but arbitrary $\Lambda$. In order for it
to represent a class in the automorphic cohomology $\Lambda$ has to
be of a specific form $\Lambda_w$ involving what is called the set
$W^P$ of Kostant representatives $w$ for $P$. At this point we are
presented with the key task which is to determine the cohomological
cuspidal automorphic representations of $L(\A)$ and the Kostant
representatives for $P$. With these data at hand it remains to
decide whether the given Eisenstein series is holomorphic or has a
pole at the point $\Lambda_w$. In solving the latter the general
Theorems available in the theory of Eisenstein cohomology provide us
with a closing answer to the initial question, which roughly assumes the following form: (see Theorems
\ref{mainthm1} \& \ref{mainthm2} as well as section \ref{regec} for the precise statements)

\begin{res}
Let $P$ be the standard maximal parabolic $\Q$-subgroup of
$G=SO_{2n+1}$ with Levi subgroup $L\cong GL_k\times SO_{2(n-k)+1}$.
Further, let $\alpha$ be the only simple root of $G$ which does not vanish
identically on $A$. Suppose given a non-trivial class in
$H^*(\g,K;\mathcal{A}(G)\otimes E)$ associated to a choice $(\pi,w)$
of a cohomological, globally generic, cuspidal automorphic
representation $\pi$ of $L(\A)$ and a Kostant representative $w\in
W^P$. As $\pi=\chi(\sigma\hat\otimes\tau)$ with $\chi$ the central
character, $\sigma$ a cohomological cuspidal automorphic
representation of $GL_k$ and $\tau$  a cohomological, globally
generic, cuspidal automorphic representation of $SO_{2(n-k)+1}$ in
the \emph{residual case} $\pi$ is either of the form
\begin{enumerate}
\item $d\chi=\Lambda_w=\frac{k}{2}\alpha$, $\sigma$ self-dual and such that the Rankin-Selberg $L$-function of $\sigma\times\tau$ does not vanish at $\frac{1}{2}$ (a condition, which we set empty if $k=n$). If in addition $k$ is even and the central character of $\sigma$ is trivial, then $k\geq 4$ and $\sigma$ is a self-dual Weak Langlands functorial lift of a globally generic cuspidal automorphic representation of $SO_k(\A)$. Or
\item $d\chi=\Lambda_w=k\alpha$, $k\neq n$ is even and $\sigma$ self-dual such that the Rankin-Selberg $L$-function of $\sigma\times\tau$ has a pole at $1$.
\end{enumerate}
The \emph{regular case} pertaining to the situation that $\pi$ either doesn't meet either one of conditions $(1)$ and $(2)$ or it does while producing regular values for the Eisenstein series is settled in light of the description of regular Eisenstein classes available in general.

Finally, we find lower and upper bounds for the degree of the Eisenstein class constructible from the above in both the residual and the regular case.
\end{res}

The structure of the paper mirrors closely the steps mentioned in
the above outline:

Section 1 introduces the automorphic cohomology of a reductive
algebraic $\Q$-group and provides a brief outline of Eisenstein
cohomology. In particular, it sketches the decomposition of the
automorphic cohomology along the cuspidal support and gathers the
main Theorems pertaining to the construction of cohomology classes
by means of Eisenstein series.

Section 2 gives a terse description of so-called Kostant data by
virtue of a method the idea behind which we learnt from N. Grbac.
Its main advantage is to be seen in a formulaic description of the
action of Kostant representatives on arbitrary weights on the Cartan
subalgebra of $\g$, and consequently of the Kostant data relevant to
the aforementioned construction.

Section 3 lists the cuspidal representations of the Levi subgroups
with non-trivial cohomology and provides lower and upper bounds for
their contribution to cohomology drawing from well-known results by
a number of people (for precise references cf. section 3).

Section 4 recounts H. Kim's determination of the part of the
residual spectrum of $G$ accounted for by the maximal parabolic
$\Q$-subgroups in order to derive the possible poles of Eisenstein
series.

Finally, sections 5 and 6 present our results on residual and
regular Eisenstein cohomology classes for maximal parabolic
$\Q$-subgroups of $G=SO_{2n+1}$.\\\\
On a final note, we want to remark that at the time of composing
this work, N. Grbac and J. Schwermer were addressing similar
questions for split $Sp_{2n}$. An analogous result exists for split
$SL_n$, see \cite{schwSLn}. The interested reader may also consult
the first author's paper \cite{gotsb}, where regular Eisenstein
cohomology classes were constructed for the inner forms $SO(n,2)$ of
$SO_{n+2}$ or the second author's papers \cite{grob22} and
\cite{grob11} on residual and regular Eisenstein cohomology of
$Sp(2,2)$ and $Sp(1,1)$.

\subsection*{Acknowledgements}
The authors would like to thank Neven Grbac and Goran Muic for
valuable discussions. We are also grateful to Dihua Jiang, Joachim Schwermer and Birgit
Speh for helpful remarks. Both authors would like to thank the Erwin Schr\"odinger Institute for Mathematical Physics, Vienna, for its hospitality, where this paper took its final form. The second named author would also like to thank the Max-Planck-Institut f\"ur Mathematik, Bonn, where the major part of this work was written.

\subsection*{Notation and Conventions} Throughout this paper $G$ will
denote a connected, reductive, but most often semisimple algebraic group
over $\Q$ of rank $rk_\Q(G)\geq 1$ with finite center. Lie algebras
of groups of real points of algebraic groups will be denoted by the
same but gothic letter, e.g. $\g=\textrm{Lie}(G(\R))$. The
complexification of a Lie algebra will be denoted by subscript
``$\C$'', e.g. $\g_\C=\g\otimes_\R\C$.\\\\ We use the standard
terminology and hypotheses concerning algebraic groups and their
subgroups to be found in \cite{moewal} I.1.4-I.1.12. In particular,
we assume that a minimal parabolic subgroup $P_0$ has been fixed and
that $K_\A=K_\R\times K_{\A_f}$ is a maximal compact subgroup of the
group $G(\A)$ of adelic points of $G$ which is in \emph{good
position} with respect to $P_0$ (\cite{moewal}, I.1.4). Then
$K=K_\R$ is maximal compact in $G(\R)$, hence has an associate
Cartan involution $\theta$. If $H$ is a
subgroup of $G$, we let $K_H=K\cap H(\R)$.\\\\
Assume that $L_0$ is the unique Levi subgroup of $P_0$ which is
invariant under $\theta$ and $N_0$ is an unipotent radical of $P_0$
such that we have the Levi decomposition $P_0=L_0N_0$. If we
additionally denote by $A_0$ a maximal, central $\Q$-split torus in
$L_0$ then we also get the Langlands decomposition $P_0=M_0A_0N_0$.
As usual, $M_0=\bigcap_\chi\ker\chi^2$, $\chi$ ranging over the
group $X(L_0)$ of all $\Q$-characters on $L_0$. Let $P$ be a
standard parabolic $\Q$-subgroup of $G$ with respect to $P_0$. It
has a unique Levi decomposition $P=L_PN_P$, with $L_P\supseteq L_0$
and also a unique Langlands decomposition $P=M_PA_PN_P$ with unique
$\theta$-stable split component $A_P\subseteq A_0$. If it is clear
from the context we will also omit the subscript ``$P$''. We write
$\Delta(P,A)$ for the set of weights of the adjoint action of $P$
with respect to $A_P$. $\rho_P$ denotes the half-sum of these
weights. In particular, $\rho=\rho_{P_0}$ is the half sum of
positive $\Q$-roots of $G$ with respect to $A_0$.\\Extend the Lie
algebra $\a$ of $A(\R)$ to a Cartan subalgebra $\h$ of $\g$ by
adding a $\theta$-stable Cartan subalgebra $\b$ of $\m$. The absolute root system of
$\g$ is denoted $\Delta=\Delta(\g_\C,\h_\C)$, a simple subsystem
(compliant with the requirement that positivity on the system of
absolute roots shall be compatible with the positivity on the set
$\Delta_\Q$ of $\Q$-roots given by the choice of the minimal pair
$(P_0,A_0)$) is denoted $\Delta^\circ$. We also write
$\Delta^\circ_M$ for the set of absolute simple roots of $\m$ with
respect to $\b$ (so $\Delta^\circ=\Delta^\circ_G$). The Weyl group
associated to $\Delta$ is denoted $W=W(\g_\C,\h_\C)$. We let
$W^P=\{w\in W|w^{-1}(\alpha)>0\quad\forall\alpha\in\Delta^\circ_M\}$. The elements of $W^P$ are called {\it Kostant representatives}.\\
Using the fact that $K_\A$ is in good position, we can extend the
standard \emph{Harish-Chandra height-function} $H_P:P(\A)\ra\a^*$
given by $\prod_{p}|\chi(p)|_p=e^{\<\chi,H_P(p)\>}$, $\chi\in X(L)$
viewed as an element of $\a_\C^*$, to a function on all of $G(\A)$
by setting $H_P(g):=H_P(p)$, $g=kp$. \\\\
Let $G$ be a connected, \emph{reductive} group over $\Q$ and $\chi$
a central character. As usual $L^2_{dis}(G(\Q)\backslash G(\A))$
(resp. $L^2_{dis}(G(\Q)\backslash G(\A),\chi)$) denotes the discrete
spectrum of $G$ (resp. the part of it consisting of functions with
central character $\chi$). It can be written as the direct sum of
the cuspidal spectrum $L^2_{cusp}(G(\Q)\backslash G(\A))$ (resp.
$L^2_{cusp}(G(\Q)\backslash G(\A),\chi)$) and the residual spectrum
$L^2_{res}(G(\Q)\backslash G(\A)$ (resp. $L^2_{res}(G(\Q)\backslash
G(\A),\chi)$). By \cite{gelf} the space $L^2_{dis}(G(\Q)\backslash
G(\A),\chi)$, decomposes as direct Hilbert sum over all irreducible,
admissible representations $\pi$ of $G(\A)$ with central character
$\chi$, each of which occurring with finite multiplicity
$m_{dis}(\pi)$. The same is therefore true for the cuspidal (resp.
residual) spectrum, if we replace the multiplicity by $m(\pi)$
(resp. $m_{res}(\pi)$). Every $\pi$ can be written as a restricted
tensor product $\pi=\otimes'_p\pi_p$, where $p$ is a place of $\Q$.
i.e. either a rational prime or $\infty$ and $\pi_p$ a local
irreducible, admissible representation $\pi_p$ of $G(\Q_p)$,
\cite{flath}. Further, $\pi$ is (and therefore all $\pi_p$ are,
simultaneously) unitary if and only if $\chi$ is. Then $\pi$ is the
completed restricted tensor product $\pi=\hat\otimes'_p\pi_p$. \\\\
For any $G(\A)$-representation $\sigma$, we will write
$\sigma^\infty$ for the space of its smooth vectors and
$\sigma_{(K)}$ for the space of $K$-finite vectors. Clearly, if
$\sigma$ is unitary, then $\sigma^\infty_{(K)}$ is a unitary
$(\g,K,G(\A_f))$-module.

\section{Generalities on Automorphic forms and Cohomology}
\subsection{}
We start our study with the space of automorphic forms $\mathcal
A(G)=\mathcal A(G(\Q)\backslash G(\A))$ on $G(\A)$. It is a
$(\g,K,G(\A_f))$-module and hence it makes sense to talk about its
$(\g,K)$-cohomology (to be called the {\it automorphic cohomology of
$G$}) which we may also twist by an irreducible rational
representation $E$ of $G(\R)$ of highest weight $\lambda$ on a
finite-dimensional, complex vector space:
$$H^q(G,E):=H^q(\g,K,\mathcal A(G)\otimes E).$$ Clearly,
$H^q(G,E)$ carries a $G(\A_f)$-module structure, induced from the action of $G(\A_f)$ on $\mathcal A(G)$, which we shall now investigate.\\\\
In order to do so, we shall analyze the cohomological automorphic
representations $\pi$ of $G$. Recall (\cite{langaut}, Prop. 2) that
a representation $\pi$ of $G(\A)$ is automorphic if and only if it
is an irreducible constituent of a globally (normalized) induced
representation Ind$^{G(\A)}_{P(\A)}[\sigma]$, $P$ being a parabolic
subgroup of $G$ and $\sigma$ a cuspidal automorphic representation
of the Levi $L_P(\A)$. Hence, cuspidal automorphic representations
of Levi subgroups $L_P$ will play a crucial role in the
determination of $H^q(G,E)$. In fact, we may divide the space of
automorphic forms into two parts, $\mathcal A(G)=\mathcal
A_{cusp}(G)\oplus\mathcal A_{Eis}(G),$ where $\mathcal A_{cusp}(G)$
is the space of cuspidal automorphic forms and $\mathcal A_{Eis}(G)$
a natural complement spanned as a representation by all irreducible
subquotients of induced representations
Ind$^{G(\A)}_{P(\A)}[\sigma]$ as above, but with $P$ proper.
Therefrom the automorphic cohomology of $G$ inherits a natural
decomposition as $G(\A_f)$-module:

\begin{equation}\label{eq:dec}
H^q(G,E)=H^q_{cusp}(G,E)\oplus H^q_{Eis}(G,E).
\end{equation}

\subsection{} Let us refine this decomposition even further. As one may guess from
the characterization of automorphic representations as subquotients
of parabolically induced representations, there should be somehow a
refinement of (\ref{eq:dec}) which involves all parabolic subgroups
$P$ of $G$ and cuspidal automorphic representations $\pi$ of
$L_P(\A)$. This is, indeed, true and we will briefly discuss this
refined decomposition as it will serve as the starting-point of our
further investigations. \\\\
First of all, $\mathcal A(G)$ admits a certain decomposition as a
direct sum with respect to the classes $\{P\}$ of associate
parabolic $\Q$-subgroups $P\subseteq G$. This relies on such a
decomposition of the space $V_G$ of $K$-finite, left
$G(\Q)$-invariant, smooth functions $f: G(\A)\ra\C$ of uniform
moderate growth, first proved by Langlands in a letter to Borel,
\cite{lang2}. See also \cite{borlabschw} Thm. 2.4:
$V_G=\bigoplus_{\{P\}}V_G(\{P\}),$ where $V_G(\{P\})$ denotes the
space of elements $f$ in $V_G$ which are negligible along $Q$ for
every parabolic $\Q$-subgroup $Q\subseteq G$, $Q\notin\{P\}$.
Putting $\mathcal A_{P}(G)=V_G(\{P\})\cap\mathcal A(G)$ we made the
first of two steps in the decomposition of $\mathcal A(G)$ alluded
to above:

$$\mathcal A(G)=\bigoplus_{\{P\}}\mathcal A_{P}(G).$$
Observe that $\mathcal A_{G}(G)\subset
V_G(\{G\})=L^2_{cusp}(G(\Q)\backslash G(\A))^\infty_{(K)}$. Hence,
we see that the following holds:

\begin{prop}
$$H^q_{cusp}(G,E)=H^q(\g,K,\mathcal A_{G}(G)\otimes E)$$ and
$$H^q_{Eis}(G,E)=\bigoplus_{\{P\},P\neq G}H^q(\g,K,\mathcal A_{P}(G)\otimes E).$$
\end{prop}

\subsection{Eisenstein series}\label{sec:eis}
We still want to take the second step in refining (\ref{eq:dec}),
meaning that we want to decompose the summands $H^q(\g,K,\mathcal A_{P}(G)\otimes E)$
involving cuspidal automorphic representations $\pi$ of $L_P(\A)$.
We refer the reader to \cite{schwfr} for details concerning this
section.\\\\ We need some technical assumptions: For $Q=LN=MAN$
associate to the standard parabolic $P$, $\varphi_{Q}$ is a finite
set of irreducible representations $\pi=\chi\widetilde\pi$ of
$L(\A)$, with $\chi: A(\R)^\circ\ra\C^*$ a continuous character and
$\widetilde\pi$ an irreducible, unitary subrepresentation of
$L^2_{cusp}(L(\Q)A(\R)^\circ\backslash L(\A))$ of $L(\A)$ whose
central character induces a continuous homomorphism
$A(\Q)A(\R)^\circ\backslash A(\A)\ra U(1)$ and whose infinitesimal
character matches the one of the dual of an irreducible
subrepresentation of $H^*(\n,E)$. This just means that
$\widetilde\pi$ is a unitary, cuspidal automorphic representation of
$L(\A)$ whose central and infinitesimal character satisfy the above
conditions. Finally, three further ``compatibility conditions'' have
to be satisfied between these sets $\varphi_{Q}$, skipped here and
listed in \cite{schwfr}, 1.2. The family of all collections
$\varphi=\{\varphi_{Q}\}$ of such finite sets is denoted $\Psi_{P}$.
\\\\Now denote $\textrm{I}_{Q,\wt\pi}=\textrm{Ind}_{Q(\A_f)}^{G(\A_f)}\textrm{Ind}_{(\l,K_L)}^{(\g,K)}
\left[\wt\pi^\infty_{(K_L)}\right]^{m(\wt\pi)}$ (unnormalized
induction). For a function $f\in \textrm{I}_{Q,\wt\pi}$,
$\Lambda\in\a_\C^*$ and $g\in G(\A)$ we consider the Eisenstein
series (formally) defined as

$$E_{Q}(f,\Lambda)(g):=\sum_{\gamma\in Q(\Q)\backslash G(\Q)}
f(\gamma g) e^{\<\Lambda+\rho_Q,H_Q(\gamma g)\>}.$$ If we set
$(\a^*)^+:=\{\Lambda\in\a_\C^*| \Re e(\Lambda)\in\rho_Q+C\}$, where
$C$ denotes the open, positive Weyl-chamber with respect to
$\Delta(Q,A)$, the series converges absolutely and uniformly on
compact subsets of $G(\A)\times (\a^*)^+$. It is known that for
fixed $\Lambda$ the function $E_{Q}(f,\Lambda)$ on $G(\A)$ is an
automorphic form there and that the map $\Lambda\mapsto
E_{Q}(f,\Lambda)(g)$ can be analytically continued to a meromorphic
function on all of $\a_\C^*$, cf. \cite{moewal} p. 140 or \cite{lang} \S7.
It is known that the singularities $\Lambda_0$ (i.e., poles) of
$E_{Q}(f,\Lambda)$ lie along certain affine hyperplanes of the form
$R_{\alpha, t}:=\{\xi\in\a_\C^*|\<\xi,\alpha\>=t\}$ for some
constant $t$ and some root $\alpha\in\Delta(Q,A)$, called
``root-hyperplanes'' (\cite{moewal} Prop. IV.1.11 (a) or \cite{lang}
p.131). Choose a normalized vector $\eta\in\a_\C^*$ orthogonal to
$R_{\alpha, t}$ and assume that $\Lambda_0$ lies on no other
singular hyperplane of $E_{Q}(f,\Lambda)$. Then define
$\Lambda_0(u):=\Lambda_0+u\eta$ for $u\in\C$. If $c$ is a positively
oriented circle in the complex plane around zero which is so small
that $E_{Q}(f,\Lambda_0(.))(g)$ has no singularities on the interior
of the circle with double radius, then
$$\textrm{Res}_{\Lambda_0}(E_{Q}(f,\Lambda)(g)):=\frac{1}{2\pi i}\int_c E_{Q}(f,\Lambda_0(u))(g)du$$
is a meromorphic function on $R_{\alpha, t}$, called the
\emph{residue} of $E_{Q}(f,\Lambda)$ at $\Lambda_0$. Its poles lie
on the intersections of $R_{\alpha, t}$ with the other singular
hyperplanes of $E_{Q}(f,\Lambda)$. By this procedure one gets a
function holomorphic at $\Lambda_0$ in finitely many steps by taking
successive residues as explained above. \\\\Now we are able to take
the desired second step in the decomposition of the $G(\A_f)$-module summand $H^q(\g,K,\mathcal A_{P}(G)\otimes E)$:
For $\pi=\chi\widetilde\pi\in\varphi_P\in\varphi\in\Psi_{P}$ let
$\mathcal A_{P,\varphi}(G)$ be the space of functions, spanned by
all possible residues and derivatives of Eisenstein series defined
via all $f\in \textrm{I}_{P,\wt\pi}$ at the value $d\chi$. It is a
$(\g,K,G(\A_f))$-module. Thanks to the functional equations (see
\cite{moewal} IV.1.10) satisfied by the Eisenstein series
considered, this is well defined, i.e., independent of the choice of
a representative for the class of $P$ (whence we took $P$ itself)
and the choice of a representation $\pi\in\varphi_P$. Finally, we
get

\begin{prop}[\cite{schwfr} Thm.s 1.4 \& 2.3; \cite{moewal} III, Thm. 2.6]\label{deceis}
There is a direct sum decomposition as $G(\A_f)$-module

\begin{equation}\label{eq:deceis}
H^q_{Eis}(G,E)=\bigoplus_{\{P\}, P\neq
G}\bigoplus_{\varphi\in\Psi_{P}}H^q(\g,K, \mathcal
A_{P,\varphi}(G)\otimes E).
\end{equation}
\end{prop}

\begin{rem}
Notice, that the second statement entails the claim that the
$G(\A_f)$-module $H^q_{Eis}(G,E)$ is generated by
derivatives and residues of cuspidal Eisenstein series associated to
$\Lambda\in\overline{C}$.
\end{rem}

\subsection{}
How does this refined decomposition (\ref{eq:deceis}) help us in
determining the $G(\A_f)$-module $H^q_{Eis}(G,E)$? It allows us to
construct classes in the Eisenstein cohomology by lifting classes
associated to cuspidal automorphic representations $\wt\pi$. In
order to have this procedure readily available we will now recall
the notion of classes of type $(\pi,w)$, $\pi\in\varphi_P$, $w\in
W^P$. \\\\Therefore, let $\pi=\chi\widetilde\pi\in\varphi_P$ and
consider the symmetric tensor algebra
$$S_\chi(\a^*)=\bigoplus_{n\geq 0} \bigodot^n\a_\C^*,$$ $\bigodot^n\a_\C^*$
being the symmetric tensor product of $n$ copies of $\a^*_\C$, as
module under $\a_\C$: Via the natural identification
$\a_\C\ira\a_\C^*$ it is an $\a_\C$-module acted upon by
$\xi\in\a_\C\cong\a_\C^*$ via multiplication with
$\<\xi,\rho_P+d\chi\>+\xi$ (within the symmetric tensor algebra).
This explains the subscript ``$\chi$''. We extend this action
trivially on $\l_\C$ and $\n_\C$ to get an action of the Lie algebra
$\p_\C$ on the Banach space $S_\chi(\a^*)$. We may also define a $ P(\A_f)$-module structure via
the rule
$$q\cdot X=e^{\<d\chi+\rho_P,H_P(q)\>}X,$$
for $q\in P(\A_f)$ and $X\in S_\chi(\a^*)$. There is
a continuous linear isomorphism
$$\textrm{Ind}_{ P(\A_f)}^{ G(\A_f)}\textrm{Ind}_{(\l,K_L)}^{(\g,K)}
\left[\widetilde\pi^\infty_{(K_L)}\otimes
S_\chi(\a^*)\right]^{m(\widetilde\pi)}\ira
\textrm{I}_{P,\widetilde\pi}\otimes S_\chi(\a^*),$$ 
so in particular one can view the right hand side as a $(\g,K, G(\A_f))$-module by
transport of structure. Doing this, it is shown in \cite{franke}, pp. 256-257, that

\begin{equation}\label{dchi}
H^q(\g,K,\textrm{I}_{P,\wt\pi}\otimes S_\chi(\a^*)\otimes E)\cong
$$$$ \bigoplus_{\substack{w\in W^P\\ -w(\lambda+\rho)|_{\a_\C}=d\chi}}
\textrm{Ind}_{P(\A_f)}^{G(\A_f)}\left[H^{q-l(w)}(\m,K_M,(\widetilde\pi_\infty)_{(K_M)}\otimes
{}^\circ F_w)\otimes\C_{d\chi+\rho_P}\otimes
\widetilde\pi_f^{\infty_f}\right]^{m(\widetilde\pi)}.
\end{equation}
Here ${}^\circ F_w$ is the irreducible, finite dimensional representation of
$M(\C)$ with highest weight $\mu_w:=w(\lambda+\rho)-\rho|_{\b_\C}$
and $\C_{d\chi+\rho_P}$ the one-dimensional, complex
$P(\A_f)$-module on which $q\in P(\A_f)$ acts by multiplication by
$e^{\<d\chi+\rho_P,H_P(q)\>}$. A non-trivial class in a summand of
the right hand side is called a cohomology class \emph{of type}
$(\pi,w)$, $\pi\in\varphi_P$, $w\in W^P$ (this notion was first introduced in \cite{schwLNM} p. 56).\\
Further, since $L(\R)\cong M(\R)\times A(\R)^\circ$,
$\widetilde\pi_\infty$ can be viewed as an irreducible, unitary
representation of $M(\R)$. Therefore, a $(\pi,w)$ type consists of
an irreducible representation $\pi=\chi\widetilde\pi$ whose unitary
part $\widetilde\pi=\widetilde\pi_\infty\hat\otimes\widetilde\pi_f$
has at the infinite place an irreducible, unitary representation
$\widetilde\pi_\infty$ of the semisimple group $M(\R)$ with
non-trivial $(\m,K_M)$-cohomology with respect to ${}^\circ F_w$.

\subsection{The Eisenstein map}\label{eisintw}
In order to construct Eisenstein cohomology classes, we start from a
class of type $(\pi,w)$. By (\ref{dchi}) we can assume that
$d\chi=-w(\lambda+\rho)|_{\a_\C}$ and that this point lies inside
the closed, positive Weyl chamber defined by
$\Delta(P,A)$.\\\\Reinterpret $S_\chi(\a^*)$ as the Banach space of
formal, finite $\C$-linear combinations of differential operators
$\frac{\partial^\nu}{\partial\Lambda^\nu}$ on the complex,
$\ell$-dimensional vector space $\a^*_\C$. It is understood that some
choice of Cartesian coordinates $z_1(\Lambda),...,z_\ell(\Lambda)$ on
$\a^*_\C$ has been fixed and $\nu=(n_1,...,n_\ell)\in\N_0^{\ell}$ denotes
a multi-index with respect to these. As a consequence of
\cite{moewal} Prop. IV.1.11, there exists a polynomial $0\neq
q(\Lambda)$ on $\a^*_\C$ such that for every $f\in
\textrm{I}_{P,\wt\pi}$ the function
$$\Lambda\mapsto q(\Lambda)E_P(f,\Lambda)$$
is holomorphic at $d\chi$. Since $\mathcal A_{P,\varphi}(G)$ can be
written as the space which is generated by the coefficient functions
in the Taylor series expansion of $q(\Lambda)E_P(f,\Lambda)$ at
$d\chi$, $f$ running through $\textrm{I}_{P,\wt\pi}$, (cf.
\cite{schwfr}) we are able to define a surjective homomorphism of
$(\g,K,G(\A_f))$-modules $E_{P,\pi}$

\begin{displaymath}
\xymatrix{ \textrm{I}_{P,\wt\pi}\otimes
S_\chi(\a^*)\ar[rr]^{E_{P,\pi}} & & \mathcal A_{P,\varphi}(G)}
$$$$
f\otimes\frac{\partial^\nu}{\partial\Lambda^\nu}\mapsto
\frac{\partial^\nu}{\partial\Lambda^\nu}\left(q(\Lambda)E_P(f,\Lambda)\right)|_{d\chi}.
\end{displaymath}
and hence get a well-defined homomorphism in cohomology:

\begin{equation}\label{con}
H^q(\g,K,\textrm{I}_{P,\wt\pi}\otimes S_\chi(\a^*)\otimes
E)\stackrel{E^q_\pi}{\longrightarrow} H^*(\g,K,\mathcal
A_{P,\varphi}(G)\otimes E).
\end{equation}
There are the following results. The first one deals with the
regular (i.e. holomorphic) case:

\begin{thm}[\cite{schwLNM}, Thm. 4.11]\label{thm:holeis}
Suppose $[\beta]\in H^q(\g,K,\textrm{\emph{I}}_{P,\wt\pi}\otimes
S_\chi(\a^*)\otimes E)$ is a class of type $(\pi,w)$, represented by
a homomorphism $\beta$, such that for all elements
$f\otimes\frac{\partial^\nu}{\partial\Lambda^\nu}$ in its image,
$E_{P,\pi}(f\otimes\frac{\partial^\nu}{\partial\Lambda^\nu})=
\frac{\partial^\nu}{\partial\Lambda^\nu}\left(q(\Lambda)E_P(f,\Lambda)\right)|_{d\chi}$
is just the regular value $E_P(f,d\chi)$ of the Eisenstein series
$E_P(f,\Lambda)$, which is assumed to be holomorphic at the point
$d\chi=-w(\lambda+\rho)|_{\a_\C}$ inside the closed, positive Weyl
chamber defined by $\Delta(P,A)$. Then $E^q_\pi([\beta])$ is a
non-trivial Eisenstein cohomology class
$$E^q_\pi([\beta])\in H^q(\g,K,\mathcal A_{P,\varphi}(G)\otimes E).$$
\end{thm}
Parts of the residual case are treated in \cite{grob22}. For sake of
simplicity we also assume that $P$ is self-associate. Put
$$\textrm{I}_{P,\widetilde\pi,\Lambda}:=\textrm{Ind}_{P(\A_f)}^{G(\A_f)}\textrm{Ind}_{(\l,K_L)}^{(\g,K)}
\left[\widetilde\pi^\infty_{(K_L)}\otimes\C_{\Lambda+\rho_P}\right]^{m(\widetilde\pi)}=\textrm{I}_{P,\widetilde\pi}\otimes\C_{\Lambda+\rho_P}$$
and recall the standard intertwining operators
$M(\Lambda,\wt\pi,v):\textrm{I}_{P,\widetilde\pi,\Lambda}\ra\textrm{I}_{P,v(\widetilde\pi),v(\Lambda)},$
see \cite{moewal}, II, associated to $\Lambda\in\a^*_\C$, $\wt\pi$
and certain Weyl group elements $v\in W(A):=N_{G(\Q)}(A(\Q))/L(\Q)$.
If $f\in \textrm{I}_{P,\wt\pi}$, we write
$f_\Lambda=fe^{\<\Lambda+\rho_P,H_P(.)\>}\in
\textrm{I}_{P,\widetilde\pi,\Lambda}$. If $M(\Lambda,\wt\pi,v)$ has a pole at $\Lambda=\Lambda_0$, then we assume to have normalized
it to a function $N(\Lambda,\wt\pi,v)$, which is holomorphic and
non-vanishing in a region containing $\Lambda_0$. Put
$$W(A)_{\textrm{res}}=\{v\in W(A)| M(\Lambda,\wt\pi,v) \textrm{ has a pole of order $\ell=\dim\a_\C$ at } \Lambda=d\chi\}.$$
This means that the order of the pole is maximal and implies that
the longest element $w_0$ of $W(A)$ (as a reduced word in the simple
reflections generating $W(A)$) will be inside $W(A)_{\textrm{res}}$.
We have the following

\begin{thm}[\cite{grob22}, Thm. 2.1]\label{thm:poles}
Let $[\beta]\in H^q(\g,K,\textrm{\emph{I}}_{P,\wt\pi}\otimes
S_\chi(\a^*)\otimes E)$ be a class of type $(\pi,w)$. If all
Eisenstein series $E_P(f,\Lambda)$, $f\otimes 1$ in the image of
$\beta$, have a pole of maximal possible order $\ell=\dim\a_\C$ at
$d\chi=-w(\lambda+\rho)|_{\a_\C}$ inside the closed, positive Weyl
chamber defined by $\Delta(P,A)$ and if
$\textrm{\emph{Im}}N(d\chi,\wt\pi,w_0)$ is a direct summand of $\sum_{v\in
W(A)_{\textrm{res}}}\textrm{\emph{Im}}N(d\chi,\wt\pi,v)$, then
$E^q_\pi([\beta])$ contributes at least in degree $q':=q+\dim
N(\R)-2l(w)$, $l(w)$ the length of $w$.
\end{thm}

\begin{rem}
Theorem \ref{thm:poles} can always be applied to non-holomorphic Eisenstein
series coming from self-associate {\it maximal} parabolic subgroups
$P$, since then $W(A)$ has exactly one non-trivial element. We
recall further that Eisenstein series associated to
non-self-associate maximal parabolic subgroups are always
holomorphic in the region $Re(\Lambda)\geq 0$. See also
\cite{moewal}.
\end{rem}
Theorem \ref{thm:poles} is complemented by the following result

\begin{thm}[\cite{rospres}, Thm. III. 1]
Let $\sigma$ be a residual, cohomological (with respect to a
non-regular coefficient module $E$) representation of $G(\A)$ which
equals (via the constant term map) the image
$\textrm{\emph{Im}}N(d\chi,\wt\pi,w_0)$ at $\Lambda=d\chi$ of the
normalization of an intertwining operator $M(\Lambda,\wt\pi,w_0)$
which has a pole of maximal order at $\Lambda=d\chi$. Suppose
further that $d\chi$ is inside the open, positive Weyl chamber
defined by $\Delta(P,A)$ and that $\wt\pi_\infty$ is a tempered
representation of $L(\R)$. If $r$ is the lowest degree in which
$\sigma$ has non-trivial $(\g,K)$-cohomology, then the image of
$H^r(\g,K,\sigma\otimes E)$ in $H^r_{Eis}(G,E)$ is non-trivial and
consists of residual Eisenstein cohomology classes.
\end{thm}

\section{Odd orthogonal groups and their maximal parabolic $\Q$-subgroups}
\subsection{}\label{basics} The main objective of this paper is to calculate Eisenstein
cohomology of the odd orthogonal group, thus pursuing the above
ideas for this case. We will focus on the subspaces
$H^q(\g,K,\mathcal A_{P,\varphi}(G)\otimes E)$ coming from a {\it
maximal} parabolic $\Q$-subgroup $P$.
%The other cases shall be part of a future work.
In this section we provide the algebraic group data for the latter. %on $SO_{2n+1}$ and their maximal parabolic $\Q$-subgroups.
\\\\
From now on $G$ denotes the split odd orthogonal group $SO_{2n+1}$
over $\Q$ of $\Q$-rank $n\geq 3$. The algebra of diagonal matrices
$\h_\C$ in $\g_\C$ determines a Cartan subalgebra as before and we
denote by $\Delta^\circ=\{\alpha_1,...,\alpha_n\}$ the set of simple
roots of $\g_\C$. If we write $\varepsilon_i$ for the linear
functional extracting the $i$-th entry of $\h_\C$, the set
$\Delta^\circ$ is given by $\alpha_i=\ep_i-\ep_{i+1}$ for $i<n$ and
$\alpha_n=\ep_n$. Now, the simple root $\alpha_k$ determines the
unique crossed Dynkin diagram
$$\begin{array}{c}
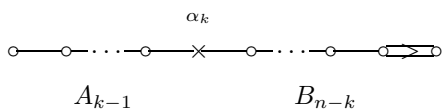
\begin{picture}(166,0)
\put(73,-2){\makebox(0,0){{\scriptsize $\alpha_k$}}}
\end{picture}\\
\begin{picture}(166,14)
\put(4,3){\line(1,0){17}} \put(25,3){\line(1,0){6}}
\put(51,3){\line(-1,0){6}} \put(54,3){\line(1,0){17}}
\put(74,3){\line(1,0){17}} \put(95,3){\line(1,0){6}}
\put(121,3){\line(-1,0){6}} \put(124,3){\line(1,0){17}}
\put(39,3){\makebox(0,0){\dots}} \put(109,3){\makebox(0,0){\dots}}
\put(144,1,2){\line(1,0){18}} \put(144,5){\line(1,0){18}}
\put(153,3){\makebox(0,0){$>$}} \put(3,2.7){\makebox(0,0){$\circ$}}
\put(23,2.7){\makebox(0,0){$\circ$}}
\put(53,2.7){\makebox(0,0){$\circ$}}
\put(73,3){\makebox(0,0){$\times$}}
\put(93,2.7){\makebox(0,0){$\circ$}}
\put(123,2.7){\makebox(0,0){$\circ$}}
\put(143,2.7){\makebox(0,0){$\circ$}}
\put(163,2.7){\makebox(0,0){$\circ$}}
\end{picture}\\
\begin{picture}(166,14)
\put(37,3){\makebox(0,0){$A_{k-1}$}}
\put(121,3){\makebox(0,0){$B_{n-k}$}}
\end{picture}
\end{array}$$
with the $k$-th node replaced by a cross, which in turn corresponds
to the unique standard maximal parabolic $\Q$-subgroup $P_k\subset
G$, $P_k=L_kN_k=M_kA_kN_k$ with Levi factor $L_k\cong GL_k\times
SO_{2l+1}$, $l=n-k$. The correspondence is by means of the
requirement that $\alpha_k$ is the only simple root that does not
vanish identically on $(\a_k)_\C$. Furthermore, since $\dim N_k =
|\Delta^+|-|\Delta^+_{M_k}|$, we see that $\dim
N_k=k(2n-k)-\binom{k}{2}$. It is easy to check that associate
classes and conjugacy classes of maximal parabolic $\Q$-subgroups
$P_k\subset G$ coincide in this case, hence
all $n$ maximal parabolic subgroups are self-associate.\\\\
There is a canonical isomorphism
$\mathfrak{h}^*\cong\mathfrak{a}^*_k\oplus\mathfrak{b}^*_k$, which
allows to restrict weights on $\mathfrak{h}_{\C}$ in a canonical way
to its direct summands. If $\beta=\sum_{i=1}^n\beta_i\ep_i$, we see
by the very definition of $\a_k$ and $\b_k$ that
$\beta|_{(\a_k)_\C}=\frac{1}{k}\sum_{i=1}^k\left(\sum_{i=1}^k\beta_i\right)\ep_i$
and hence
$\beta|_{(\b_k)_\C}=\sum_{i=1}^n\beta_i\ep_i-\frac{1}{k}\sum_{i=1}^k\left(\sum_{i=1}^k\beta_i\right)\ep_i$.
Moreover, we get
$\rho_{P_k}=\frac{k(2n-k)}{2}\alpha_k|_{(\a_k)_\C}$.

\subsection{Kostant data}
Let us now turn to the Kostant representatives $w\in W^{P_k}$. In
particular, we shall calculate the evaluation points
$\Lambda_w:=-w(\lambda+\rho)|_{(\a_k)_\C}$ of Eisenstein series
$E_{P_k}(f,\Lambda)$. Rereading Lemma $4.3$ in \cite{tadic} in view
of this latter calculation we prefer to identify the elements $w\in
W^{P_k}$ in the form set forth by the following

\begin{prop}
For each $k$, $1\leq k\leq n$, the Kostant representatives $W^{P_k}$
are parameterized by the set $\mathcal S_k$ of all ordered pairs
$(I,J)$ of disjoint subsets $I$, $J$ of $\N_{\leq n}=\{1,2,...,n\}$
satisfying $|I|+|J|=k$. A parametrization is given as follows: Let
$\i=|I|$ and $\j=|J|$, so we can write $I=\{i_1,...,i_\i\}$,
$J=\{j_1,...,j_\j\}$ and $R:=\N_{\leq n}\backslash I\cup
J=\{r_1,...,r_{n-k}\}$. Then the element $\wij\in W^{P_k}$
corresponding to the pair $(I,J)$ is given by\\
$w_{(I,J)}(\ep_{i_l}):=-\ep_{k+1-l}$ for $i_l\in I$,\\
$w_{(I,J)}(\ep_{j_l}):=\ep_{l}$ for $j_l\in J$ and\\
$w_{(I,J)}(\ep_{r_l}):=\ep_{k+l}$ for $r_l\in R$.
\end{prop}
\begin{proof}
First of all we notice that $|W^{P_k}|=2^k\binom{n}{k}=|\mathcal
S_k|$. So we only need to show that $\wij\in W^{P_k}$. But since

$$\wij^{-1}(\alpha_l)=\left\{
\begin{array}{ll}
\ep_{j_l}-\ep_{j_{l+1}} & 1\leq l\leq\j-1\\
\ep_{j_\j}-\ep_{i_\i} & l=\j\\
\ep_{i_{k-l}}-\ep_{i_{k-l+1}} & \j+1\leq l\leq k-1\\
\ep_{r_{l-k}}-\ep_{r_{l-k+1}} & k+1\leq l\leq n-1\\
\ep_{r_{n-k}} & l=n\\
\end{array}
\right.$$ $\wij\in W^{P_k}$ by the very definition of $W^{P_k}$.
\end{proof}

\begin{rem}
The description of $w\in W^{P_k}$ as in the Proposition is seen to
amount to the one given in \cite{tadic} by observing that
$X^k_j=\{\wij|\j=j\}$ (in the notation of \cite{tadic}), where the
$k$ here corresponds to the $i$ there.
\end{rem}
Writing $\wij$ as a word in the simple reflections, the next Lemma
is immediate.

\begin{lem}\label{length}
Let $m:=\max(\{l:j_l<i \quad\forall i\in I\}\cup\{0\})$. Then the
length of $\wij$ is
$$l(\wij)=\sum_{l=1}^\i(2n-k-i_l+1)+\sum_{l=1}^\j(j_l-l)-\sum_{l=m+1}^\j|\{i\in I:i<j_l\}|.$$
\end{lem}
As announced in the beginning of this section, we want to determine
the evaluation points $\Lambda_w=-w(\lambda+\rho)|_{(\a_k)_\C}$. In
what follows, we will write $\lambda=\sum_{i=1}^n\lambda_i\ep_i$.
Using our parametrization of the Kostant representatives a straight
forward computation shows

\begin{prop}\label{sws}
$$-\wij(\lambda+\rho)|_{(\a_k)_\C}=\left(\sum_{l=1}^\i(\lambda_{i_l}-i_l)-\sum_{l=1}^\j(\lambda_{j_l}-j_l)+(\i-\j)(n+\frac{1}{2})\right)\alpha_k|_{(\a_k)_\C}$$
\end{prop}
Let us write $t_{(I,J)}$ for the above coefficient of $\alpha_k$ in
$-\wij(\lambda+\rho)|_{(\a_k)_\C}$. Then $t_{(I,J)}$ is always a
half-integer. Now we compute the highest weights
$\mu_{\wij}=\wij(\lambda+\rho)-\rho|_{(\b_k)_\C}$ of the irreducible
$M(\C)$-modules ${}^\circ F_{\wij}$ by subtracting
$\wij(\lambda+\rho)-\rho|_{(\a_k)_\C}=-\left(\frac{1}{k}t_{(I,J)}+n-\frac{k}{2}\right)\sum_{l=1}^k\ep_l$
from $\wij(\lambda+\rho)-\rho$, cf. section \ref{basics}.

\begin{prop}\label{muws}
We have
\begin{eqnarray*}
\mu_{\wij} & = & \sum_{l=1}^\j(\lambda_{j_l}-j_l+l+\frac{1}{k}t_{(I,J)}+n-\frac{k}{2})\ep_l\\
& & -\sum_{l=1}^\i(\lambda_{i_{\i-l+1}}-i_{\i-l+1}-\j-l+n+1-\frac{1}{k}t_{(I,J)}+\frac{k}{2})\ep_{\j+l}\\
& & + \sum_{l=1}^{n-k}(\lambda_{r_l}-r_l+k+l)\ep_{k+l}
\end{eqnarray*}
\end{prop}
Next, we recall that we may assume that ${}^\circ F_w$ is isomorphic
to its contragredient representation ${}^\circ\check{F}_w$. This is
due to \cite{bocas}, where it is proved that the existence of a
square-integrable automorphic representation of $L_k(\A)$ which is
cohomological with respect to ${}^\circ F_w$ implies that ${}^\circ
F_w$ is self-dual. In particular, if ${}^\circ F_w$ is not
self-dual, then there is no cuspidal automorphic representation
$\wt\pi$ which has non-trivial cohomology when twisted by ${}^\circ
F_w$. \\\\The finite-dimensional representation ${}^\circ F_w$ being
self-dual is equivalent to

\begin{equation}\label{selfdual}
-w_{L_k}(\mu_w)=\mu_w,
\end{equation}
where we wrote $w_{L_k}$ for the longest element of the Weyl group
of $L_k(\C)$. By prop. \ref{muws}, we see that

\begin{eqnarray*}
-w_{L_k}(\mu_{\wij}) & = & -\sum_{l=1}^\j(\lambda_{j_l}-j_l+l+\frac{1}{k}t_{(I,J)}+n-\frac{k}{2})\ep_{k-l+1}\\
& & + \sum_{l=1}^\i(\lambda_{i_{\i-l+1}}-i_{\i-l+1}-\j-l+n+1-\frac{1}{k}t_{(I,J)}+\frac{k}{2})\ep_{k-\j-l+1}\\
& & + \sum_{l=1}^{n-k}(\lambda_{r_l}-r_l+k+l)\ep_{k+l}
\end{eqnarray*}
Assume that $\i<\j$ and (\ref{selfdual}) holds. Then comparing the
coefficient of $\ep_{\i+1}$ in $-w_{L_k}(\mu_w)$ to the coefficient
of $\ep_{\i+1}$ in $\mu_w$ leads to the equality

\begin{equation}\label{eq:hm}
\lambda_{j_{\i+1}}+\lambda_{j_{\j}}+\frac{2}{k}t_{(I,J)}+2n+1=j_{\i+1}+j_{\j}.
\end{equation}
As remarked at the end of section \ref{sec:eis}, we may assume by
the work of J. Franke that $t_{(I,J)}\geq 0$. But then the left hand
side in (\ref{eq:hm}) is greater or equal to $2n+1$, while the right
hand side is at most $2n$. This is a contradiction. So we may assume
from now on that $\i\geq\j$. \\\\In order to make the determination
of poles of Eisenstein series as simple and efficient as possible,
we shall try to find restrictions on the range of evaluation points.
In this sense, the following Proposition will be crucial for us.

\begin{prop}\label{nows}
There is no $w=\wij\in W^{P_k}$ satisfying
$-w_{L_k}(\mu_{w})=\mu_{w}$ and $0\leq t_{(I,J)}<\frac{k}{2}$.
\end{prop}
\begin{proof}
First assume $\j> 0$. Suppose that we found a $w=\wij\in W^{P_k}$
satisfying $-w_{L_k}(\mu_{w})=\mu_{w}$ and $0\leq
t_{(I,J)}<\frac{k}{2}$. Then equation (\ref{selfdual}) implies that
the coefficient of $\ep_{1}$ in $-w_{L_k}(\mu_w)$ and the
coefficient of $\ep_{1}$ in $\mu_w$ must be equal and since $\j>0$,
this reads as
$$(\lambda_{i_1}-\lambda_{j_1})-(i_1-j_1)=\frac{2}{k}t_{(I,J)},$$
implying
$$0\leq (\lambda_{i_1}-\lambda_{j_1})-(i_1-j_1)<1.$$
If $j_1>i_1$ then $\lambda_{i_1}\geq\lambda_{j_1}$, leading to
$(\lambda_{i_1}-\lambda_{j_1})-(i_1-j_1)\geq 1$. But if $j_1<i_1$
then $\lambda_{i_1}\leq\lambda_{j_1}$ and this yields
$(\lambda_{i_1}-\lambda_{j_1})-(i_1-j_1)\leq -1$. A contradiction.
\\Now assume $\j=0$. Then $\i=k$ and by prop. \ref{muws} we see that
$$t_{(I,J)}=\sum_{l=1}^k(\lambda_{i_l}-i_l)+k(n+\frac{1}{2})\geq\frac{k}{2}.$$ This proves the claim.
\end{proof}
We shall also see now that there are only very few Kostant
representatives $w=\wij$ giving rise to the lowest possible,
positive point $\Lambda_w=\frac{k}{2}\alpha|_{(\a_k)_\C}.$

\begin{prop}\label{1/2}
Suppose $-w_{L_k}(\mu_{\wij})=\mu_{\wij}$ and $t_{(I,J)}=\frac{k}{2}$. Then, depending on the parity of $k$, \\
$I=\{i_1,...,i_\i=n\}$, $J=\{i_1+1,...,i_{\i-1}+1\}$,
$\lambda_{i_l}=\lambda_{i_l+1}$, $1\leq l\leq \i-1$ and
$\lambda_n=0$ if $k$ is odd,\\ $I=\{i_1,...,i_{\frac{k}{2}}\}$,
$J=\{i_1+1,...,i_{\frac{k}{2}}+1\}$,
$\lambda_{i_l}=\lambda_{i_l+1}$, $1\leq l\leq \frac{k}{2}$ if $k$ is
even. In particular the length of such an $\wij$ is unique and given
by
$$l(\wij)=\left\{
\begin{array}{ll}
\frac{k-1}{2}(2n-\frac{3(k-1)}{2})+(n-k+1) & \textrm{ if $k$ is odd}\\
\frac{k}{2}(2n-\frac{3k}{2}+1) & \textrm{ if $k$ is even}.\\
\end{array}
\right.$$
\end{prop}
\begin{proof}
Recall $\i\geq\j$. Comparing the coefficients of $\ep_l$, $1\leq
l\leq\j$, in $-w_{L_k}(\mu_{\wij})$ and in $\mu_{\wij}$ gives us as
in the proof of Proposition \ref{nows}
$$(\lambda_{i_l}-\lambda_{j_l})-(i_l-j_l)=\frac{2}{k}t_{(I,J)}=1.$$
Therefore $j_l=i_l+1$ and $\lambda_{i_l}=\lambda_{i_l+1}$ for $1\leq
l\leq \j$. This shows the claim, if $\i=\j$ (which forces $k$ to be
even). So assume that $\i>\j$. Inserting $j_l=i_l+1$ and
$\lambda_{i_l}=\lambda_{i_l+1}$ into the formula for $t_{(I,J)}$
given in Proposition \ref{sws}, yields
$$\sum_{l=\j+1}^\i i_l-n(\i-\j)=\sum_{l=\j+1}^\i\lambda_{i_l}\geq 0.$$
But the left hand side of this equation is less or equal to $0$,
with equality if and only if $\i=\j+1$ and $i_\i=n$.\\The formula
for the length of $\wij$ is now a direct consequence of Lemma
\ref{length}. Hence the claim.
\end{proof}
Assume for the rest of this section that $k<n$ and that $k$ is even.
As we will see in section \ref{k<n}, we need to know which
$w=\wij\in W^{P_k}$ give rise to $\Lambda_w=k\alpha|_{(\a_k)_\C}$
for such $k$. There are more possibilities than in the case of
$\frac{k}{2}$ and we classify them in the next Proposition. We omit
the technical proof, as it is completely analogous to the proof of
Proposition \ref{1/2}

\begin{prop}\label{1}
Suppose $-w_{L_k}(\mu_{\wij})=\mu_{\wij}$ and $t_{(I,J)}=k<n$ is even. Then, one of the following holds, \\

\begin{enumerate}
\item[(i)] $I=\{i_1,...,i_{\i-1}=n-1,i_\i=n\}$, $J=\{i_1+1,...,i_{\i-2}+1\}$, $\lambda_{i_l}=\lambda_{i_l+1}+1$, $1\leq l\leq \i-2$, $\lambda_{n-1}=\lambda_n=0$ and
$$l(\wij)=k(n-\frac{3k}{4}+\frac{1}{2})+1,$$\item[(ii)]
$I=\{i_1,...,i_{\i-1}=n-1,i_\i=n\}$, $J=\{i_1+2,...,i_{\i-2}+2\}$,
$\lambda_{i_l}=\lambda_{i_l+1}$, $1\leq l\leq \i-2$ and
$$k(n-\frac{3k}{4}+1)-\lfloor\frac{k-2}{4}\rfloor\leq l(\wij)\leq
k(n-\frac{3k}{4}+1),$$
\item[(iii)] $I=\{i_1,...,i_{\frac{k}{2}}\}$, $J=\{i_1+1,...,i_{\frac{k}{2}}+1\}$, $\lambda_{i_l}=\lambda_{i_l+1}+1$, $1\leq l\leq \frac{k}{2}$ and
$$l(\wij)=k(n-\frac{3k}{4}+\frac{1}{2}),$$\item[(iv)]
$I=\{i_1,...,i_{\frac{k}{2}}\}$,
$J=\{i_1+2,...,i_{\frac{k}{2}}+2\}$,
$\lambda_{i_l}=\lambda_{i_l+1}$, $1\leq l\leq \frac{k}{2}$ and
$$k(n-\frac{3k}{4}+1)-\lfloor\frac{k}{4}\rfloor\leq l(\wij)\leq
k(n-\frac{3k}{4}+1),$$
\end{enumerate}
In any case,

\begin{equation}\label{ineq1}
k(n-\frac{3k}{4}+\frac{1}{2})\leq l(\wij)\leq k(n-\frac{3k}{4}+1)
\end{equation}
\end{prop}

\section{Cohomological cuspidal representations}

\subsection{} As a second ingredient to Eisenstein cohomology we shall determine the {\it
cohomological} (unitary) cuspidal automorphic representations
$\wt\pi\in\varphi_{P_k}$ of $L_k(\A)$. It is clear that
$\wt\pi=\sigma\hat\otimes\tau$, where $\sigma$ (resp. $\tau$) is a
cohomological, unitary cuspidal automorphic representation of
$GL_k(\A)$ (resp. $SO_{2l+1}(\A)$). Further, a representation is
cohomological if and only if its infinite component is. Let us first
consider the $GL_k$-factor.

\subsection{}
Recall that $\sigma_\infty$ is actually a representation of the
semisimple part of $GL_k(\R)$, which is $SL_k^\pm(\R)$. If $k=1$
then $\sigma_\infty$ must be the same character as the one of the
coefficient module in cohomology and if $k=2$, it must be a discrete
series representation. So suppose $k>2$. The cohomological
irreducible, unitary representations $\sigma_\infty$ of
$SL_k^\pm(\R)$, $k>2$, are implicitly classified in \cite{speh} Thm.
4.2.2. Let us recall this result shortly for the case of generic
representations. This is no restriction, since cuspidal automorphic
representations of $GL_k(\A)$, $k\geq 2$, are all globally generic
(cf. \cite{shal} corollary on p. 190). Let
$a=\lfloor\frac{k}{2}\rfloor$, $b=k-2a$ and put
$M(\R)=\prod_{i=1}^{a} SL^\pm_{2}(\R)\times\{\pm1\}^b$. Then we can
say

\begin{prop}[\cite{speh} Thm. 4.2.2]
If $\sigma_\infty$ is a generic, cohomological, irreducible, unitary
representation of $SL^\pm_k(\R)$, $k>2$, then
$$\sigma_\infty=\textrm{\emph{Ind}}^{SL_k^\pm(\R)}_{P(\R)}[\sigma_1\otimes...\otimes\sigma_{a+1}\otimes\C_{\rho_P}].$$
Here $P(\R)=M(\R)A(\R)^\circ N(\R)$ is a parabolic subgroup of
$SL_k^\pm(\R)$ with semisimple part $M(\R)$ as above and $\sigma_i$
is either the trivial representation or a discrete series
representation of $SL_2^\pm(\R)$.
\end{prop}
This Proposition together with our above considerations implies the
following remarkable fact: A generic, cohomological, irreducible,
unitary representation $\sigma_\infty$ of $SL_k^\pm(\R)$, $k\geq 2$,
is necessarily tempered (because it is fully and unitarily induced
from a discrete series representation). Hence, by \cite{bowa} III,
Prop. 5.3, we can conclude

\begin{prop}\label{deggl}
Let $\sigma$ be a cuspidal automorphic representation of $GL_k(\A)$,
$k\geq 1$, as above. If $\sigma$ has non-zero cohomology in degree
$q$ then
$$\frac{1}{2}\left(\frac{k(k-1)}{2}+\lfloor\frac{k}{2}\rfloor\right)\leq q\leq
\frac{1}{2}\left(\frac{(k-1)(k+4)}{2}-\lfloor\frac{k}{2}\rfloor\right).$$
\end{prop}

See also \cite{schwSLn} Thm. 3.3 for a similar result.

\subsection{}
Let us now turn to the $SO_{2l+1}$-factor and a cohomological,
cuspidal automorphic representation $\tau$ of it. We suppose that
$\tau_\infty$ is locally generic. Combining Kostant's
characterization of generic Harish-Chandra modules in
\cite{kostantgeneric} with Vogan's description of large Harish-Chandra modules in \cite{vogan} Thm. 6.2, we see that
again $\tau_\infty$ must be induced from a discrete series
representation. Furthermore, the Gelfand-Kirillov dimension (cf.
\cite{vogan} for a definition)
%i.e. the degree of the Hilbert polynomial of $\tau_\infty$ as a module for the associated graded of the universal enveloping algebra of $so(2l+1)$
of $\tau_\infty$ must be maximal; that is equal to $l^2$ in our
present case. On the other hand $\tau_\infty$ is cohomological,
whence it is an $A_\q(\lambda)$-module in the sense of Vogan and
Zuckerman (\cite{vozu}). By checking which $A_\q(\lambda)$-modules
of $SO(l+1,l)^\circ(\R)$ are actually of Gelfand-Kirillov dimension
$l^2$, we see that $\tau_\infty$ must be a discrete series
representation. Hence we have proved

\begin{prop}
If $\tau_\infty$ is a generic, cohomological, irreducible unitary
representation of $SO(l+1,l)(\R)$ then it is in the discrete series.
\end{prop}

By \cite{bowa} II, Thm. 5.4 we therefore have

\begin{prop}\label{degso}
Let $\tau$ be a cuspidal automorphic representation of
$SO_{2l+1}(\A)$ as above. If $\tau$ has non-zero cohomology in
degree $q$ then
$$q=\frac{l^2+l}{2}.$$
\end{prop}

\subsection{}
Putting our Propositions \ref{deggl} and \ref{degso} together we can
finally conclude by the K\"unneth rule (\cite{bowa} 1.3) the
following

\begin{thm}\label{boundsq}
Let $\wt\pi\in\varphi_{P_k}$ be a cuspidal automorphic
representation of $L_k(\A)$ having a generic archimedean component
$\wt\pi_\infty$. If $\wt\pi$ has non-trivial cohomology in degree
$q$ then
$$\frac{1}{2}\left(\frac{k(k-1)}{2}+\lfloor\frac{k}{2}\rfloor+l^2+l\right)\leq q\leq
\frac{1}{2}\left(\frac{(k-1)(k+4)}{2}-\lfloor\frac{k}{2}\rfloor+l^2+l\right).$$
\end{thm}

\section{Poles of Eisenstein series}\label{sec:poles}

\subsection{}
We will now calculate the possible poles of our Eisenstein series.
By the Langlands ``Square-Integrability Criterion'', \cite{moewal},
Lemma I.4.11, this amounts to determining the part of the residual
spectrum of $G(\A)=SO_{2n+1}(\A)$ which is given by the maximal
parabolic subgroups $P_k$, a task that was achieved by H. Kim in
\cite{kim01}. For the convenience of the reader and for sake of
completeness of our presentation we repeat his arguments briefly.
Still our results here will be in a somewhat different guise.\\\\
Let now $\pi=\hat\otimes'\pi_p$ be a (cohomological) globally {\it
generic} cuspidal automorphic representation of a maximal Levi
subgroup $L_k(\A)=GL_k(\A)\times SO_{2l+1}(\A)$, $1\leq k\leq n$,
$l=n-k$. As already remarked earlier, such a representation is
necessarily of the form $\pi=\sigma\hat\otimes\tau$, where $\sigma$
is a cuspidal representation of $GL_k(\A)$ and $\tau$ a generic,
cuspidal representation of $SO_{2l+1}(\A)$. It enjoys Strong
Multiplicity One, combining the results Thm. 4.4 in \cite{jacsha},
and Thm. 9 in \cite{ginzralsou}. Hence, by the multiplicity one
Theorem for $GL_k$, $m(\pi)=m(\tau)$ holds and following Arthur's
Conjecture, even $m(\tau)=1$ should be true. We will not assume
this. We identify $\Lambda=t\alpha_k\in(\a_k)^*_\C$ with
$s=\frac{t}{k}\in\C$ if $k<n$ and with $s=\frac{2t}{n}\in\C$ if
$k=n$, following \cite{shahidi2}, p. 552. We will also omit the subscript ''$k$''.\\\\
Let $f\in \textrm{I}_{P,\pi}$ then
$f_s=fe^{(s+\rho_P)H_P(.)}\in\textrm{I}_{P,\pi,s}$. The holomorphic
behavior of the Eisenstein series $E_P(f,s)$ is the same as the one
of its constant term along $P$ (cf. \cite{lang}, \cite{moewal},
IV.1.10), which can be rewritten as

$$E_P(f,s)_P=f_s+M(s,\pi)f_s.$$
Here $M(s,\pi)$ is the standard intertwining operator (cf. section
\ref{eisintw}) of $(\g,K,G(\A_f))$-modules
$$M(s,\pi):\textrm{I}_{P,\pi,s}\ra\textrm{I}_{P,w_0(\pi),-s}$$
$$M(s,\pi)f_s(g)=\int_{N(\Q)\cap w_0N(\Q)w_0^{-1}\backslash
N(\A)}f_s(w_0^{-1}ng)dn,$$ $w_0$ the only non-trivial element in
$W(A)=N_{G(\Q)}(A(\Q))/L(\Q)$. Therefore the poles of $E_{P}(f,s)$
are determined by the interplay of the poles and zeros of
$M(s,\pi)$. By twisting $\pi$ by an appropriate imaginary power of
the absolute value of the determinant we may and will assume that
all poles are real, that is $s=\Re e(s)$ in the sequel. \\\\Let $S$
be the finite set of places containing the archimedean one and the
places where $\pi$ ramifies. Using the Langlands-Shahidi method (cf.
\cite{lang3, shahidi3, shahidi2}) we see that for suitably
normalized, $L(\Z_p)$-fixed functions $\widetilde f_{s,p}$

$$M(s,\pi)f_s=\bigotimes_{p\in S}A(s,\pi_p)f_{s,p}\otimes\prod_{p\notin S}\frac{L(s,\sigma_p\times\tau_p)L(2s,\sigma_p,\mathrm{Sym}^2)}{L(1+s,\sigma_p\times\tau_p)L(1+2s,\sigma_p,\mathrm{Sym}^2)}\widetilde f_{s,p}.$$
if $k<n$ and

$$M(s,\pi)f_s=\bigotimes_{p\in S}A(s,\pi_p)f_{s,p}\otimes\prod_{p\notin S}\frac{L(s,\sigma_p,\mathrm{Sym}^2)}{L(1+s,\sigma_p,\mathrm{Sym}^2)}\widetilde f_{s,p}.$$
if $k=n$. Here $L(s,\sigma_p\times\tau_p)$ is the (local)
Rankin-Selberg $L$-function associated to
$\pi_p=\sigma_p\hat\otimes\tau_p$ and $L(s,\sigma_p,\mathrm{Sym}^2)$
is the (local) symmetric square $L$-function of $\sigma_p$. For
convenience, we distinguish the cases $k<n$ and $k=n$ in what
follows.

\subsection{$k<n$}\label{k<n} Shahidi defined the $L$-functions $L(s,\sigma_p\times\tau_p)$ and $L(s,\sigma_p,\mathrm{Sym}^2)$ also for the places $p\in S$. By \cite{kim01}, Prop. 4.1,
$$N(s,\pi_p)=\frac{L(1+s,\sigma_p\times\tau_p)L(1+2s,\sigma_p,\mathrm{Sym}^2)}{L(s,\sigma_p\times\tau_p)L(2s,\sigma_p,\mathrm{Sym}^2)}A(s,\pi_p)$$
is holomorphic and non-zero for $s\geq 0$. Therefore

\begin{prop}
There is an $f\in \textrm{\emph{I}}_{P,\pi}$ such that the Eisenstein
series $E_P(f,s)$ has a pole at $s=s_0$, $s_0>0$, if and only if
$$\frac{L(s,\sigma\times\tau)L(2s,\sigma,\mathrm{Sym}^2)}{L(1+s,\sigma\times\tau)L(1+2s,\sigma,\mathrm{Sym}^2)}$$
has a pole at $s=s_0$.
\end{prop}

It is well-known (\cite{kim01,kim08}) that $L(s,\sigma\times\tau)$
is meromorphic with possible poles only at $s=0,1$ and non-vanishing
for $s>1$. Similarly, $L(s,\sigma,\mathrm{Sym}^2)$ is holomorphic
for $s\geq 1$ (except possibly at $s=1$) and non-zero there. Hence,
the poles of $M(s,\pi)$ in $s>0$ are the ones of

\begin{equation}\label{L}
L(s,\sigma\times\tau)L(2s,\sigma,\mathrm{Sym}^2)
\end{equation}
and so - by what we just observed - we conclude that the only
possible poles of Eisenstein series $E_P(f,s)$ in the region
$s\geq\frac{1}{2}$ are at $s=\frac{1}{2},1$. This is enough for us,
as we will only need to consider Eisenstein series at
$s\geq\frac{1}{2}$, cf. Proposition \ref{nows}. Just for
completeness, let us remark that it is not known, if
$L(2s,\sigma,\mathrm{Sym}^2)$ has poles in the remaining region
$0<s<\frac{1}{2}$. But it is shown in Cor. 3.2 of \cite{neven} that
this is not the case, i.e., $L(2s,\sigma,\mathrm{Sym}^2)$ is
holomorphic for $0< s <\frac{1}{2}$, if Arthur's Conjecture as
formulated in section 30 of \cite{arthur} (see also section 2 of
\cite{neven} for a precise reformulation adapted to this purpose) on
the discrete spectrum holds. However, we will not need this.
%Neven zeigt dies nur für k>1. Aber für k=1 ist L(s,\chi,\mathrm{Sym}^2)=L(s,\triv),
%die normale Hecke-L-Funktion. Sie hat Pole nur bei s=0,1, damit stimmt
%das Resultat auch in diesem Fall!
\\\\Let us render the above more precise. If $\sigma$ is not self-dual,
then both $L$-functions in (\ref{L}) are entire. So let us from now on assume that $\sigma$ is self-dual.\\\\
Let us first consider the case $s=\frac{1}{2}$. Then the pole can
originate only from the symmetric square $L$-function. If either $k$
is odd or the central character $\omega_\sigma$ of $\sigma$ is
non-trivial, then this is the case, i.e.,
$L(2s,\sigma,\mathrm{Sym}^2)$ has a pole at $s=\frac{1}{2}$. This is
well-known and can be seen as follows: The statement for odd $k$ is
a consequence of the Rankin-Selberg convolutions of either
\cite{bumfrie} or \cite{kazpat}. If $k$ is even, but
$\omega_\sigma\neq\triv$, then the corresponding assertion is proved
in Prop. 3.7 of \cite{kim01}. So, let now $k=2$ and
$\omega_\sigma\equiv\triv$. Then $L(2s,\sigma,\mathrm{Sym}^2)$ is
holomorphic at $s=\frac{1}{2}$. This is clear by the following easy
consideration: Recall that

\begin{equation}\label{eq:LLL}
L(s,\sigma\times\sigma)=L(s,\sigma,\mathrm{Sym}^2)L(s,\sigma,\wedge^2)
\end{equation}
has a simple pole at $s=1$, because $\sigma$ is supposed to be
self-dual. We also may replace all $L$-functions by their partial
analogues with respect to the set $S$ without changing this
assertion, since $\sigma_\infty$ is as cohomological representation
of discrete series, whence all $\sigma_p$ are tempered
(\cite{delWeil}, Thm. I.6, together with \cite{deligneW-R}, Thm.
5.6) and so the local $L$-functions of $\sigma_p$ are holomorphic
and non-vanishing for $s>0$. But for $p\notin S$,
$L(s,\sigma_p,\wedge^2)=L(s,\omega_{\sigma_p})$, which has a pole at
$s=1$ as $\omega_\sigma\equiv\triv$ by assumption and so together
with (\ref{eq:LLL}) we see that $L(s,\sigma,\mathrm{Sym}^2)$ must be
holomorphic at $s=1$. \\\\In order to understand the situation
better in the case of even $k\geq 4$ and $\omega_\sigma\equiv\triv$,
we reformulate it in terms of the Weak Langlands Functoriality. The
references for this are \cite{ckpss}, \cite{ginzralsou}, \cite{gold}
and \cite{soudry}. Suppose that $k\geq 4$ is even and that the
central character of $\sigma$ is trivial. In this situation we know
by automorphic descent that if $L(2s,\sigma,\mathrm{Sym}^2)$ has a
pole at $s=\frac{1}{2}$, then $\sigma$ is the Weak Langlands
Functorial Lift from a globally generic cuspidal automorphic
representation of $SO_k(\A)$. On the other hand, if
$L(2s,\sigma,\mathrm{Sym}^2)$ is holomorphic at $s=\frac{1}{2}$,
then automorphic descent tells us that $\sigma$ is the Weak
Langlands Functorial Lift from a globally generic cuspidal
automorphic representation of $SO_{k+1}(\A)$. We may conclude that
$\sigma$ will therefore not be a Weak Langlands Functorial Lift from
a cuspidal automorphic representation of $SO_k(\A)$.
%Siehe Cogdell, P-S, Shahidi paper im Honour-Band von Shahidi, S. 19
We remark that $s=\frac{1}{2}$ means $t=\frac{k}{2}$ and get the
following

\begin{prop}[Poles for $s=\frac{1}{2}$]\label{s=1/2}
There is an $f\in \textrm{\emph{I}}_{P,\pi}$ such that the Eisenstein
series $E_P(f,s)$ has a pole at $s=\frac{1}{2}$ if and only if
\begin{enumerate}
\item (In case $k\geq 4$ is even and $\omega_\sigma\equiv\triv$): $\sigma$ is a self-dual Weak Langlands functorial lift of a globally generic cuspidal automorphic representation of $SO_k(\A)$ and $L(\frac{1}{2},\sigma\times\tau)\neq 0$.
\item (In case $k$ is odd or $\omega_\sigma\neq\triv$): $\sigma$ is self-dual and $L(\frac{1}{2},\sigma\times\tau)\neq 0$.
\end{enumerate}
\end{prop}
The case $s=1$ is similar. There the pole can originate only from
the Rankin-Selberg $L$-function. It is entire, if $k$ is not even,
\cite{kim01}. But we know that $L(2,\sigma,\mathrm{Sym}^2)\neq 0$.
Hence we get the following

\begin{prop}[Poles for $s=1$]\label{s=1}
There is an $f\in \textrm{\emph{I}}_{P,\pi}$ such that the Eisenstein
series $E_P(f,s)$ has a pole at $s=1$ if and only if $\sigma$ is
self-dual, $k$ is even and $L(s,\sigma\times\tau)$ has a pole at
$s=1$.
\end{prop}
Observe that $s=1$ corresponds to $t=k$.

\subsection{$k=n$} The remaining case $k=n$ is treated in complete
analogy to the previous one. We only have to observe that there is
no Rankin-Selberg $L$-function appearing in the normalization
factor. Again,

$$N(s,\pi_p)=\frac{L(1+s,\sigma_p,\mathrm{Sym}^2)}{L(s,\sigma_p,\mathrm{Sym}^2)}A(s,\pi_p)$$
is holomorphic and non-vanishing for $s>0$, see \cite{kim01}, Prop.
4.1, so the poles of $M(s,\pi)$ with $P=P_n$ and $s>0$ are
determined by the holomorphic behavior of the quotient
$$\frac{L(s,\sigma,\mathrm{Sym}^2)}{L(1+s,\sigma,\mathrm{Sym}^2)}.$$
As observed before, $L(1+s,\sigma,\mathrm{Sym}^2)$ is holomorphic
and non-vanishing at $s>0$, so we have to analyze
$L(s,\sigma,\mathrm{Sym}^2)$. If $k=n$, it will be enough for us to
consider the region $s\geq 1$, see Proposition \ref{nows}. Hence we
can deduce that the only possible pole of an Eisenstein series which
interests us can be at $s=1$. As our parameter $s$ comes from
$\Lambda=\frac{ns}{2}\alpha_n$ in the case $k=n$, we see that
actually $t$ must be $\frac{n}{2}$. We have therefore proved

\begin{prop}[Poles for $s=1$]\label{s=1n}
There is an $f\in \textrm{\emph{I}}_{P,\pi}$ such that the Eisenstein
series $E_P(f,s)$ has a pole at $s=1$ if and only if
\begin{enumerate}
\item (In case $n\geq 4$ is even and $\omega_\sigma\equiv\triv$): $\sigma$ is a self-dual Weak Langlands functorial lift of a globally generic cuspidal automorphic representation of $SO_n(\A)$.
\item (In case $n$ is odd or $\omega_\sigma\neq\triv$): $\sigma$ is self-dual.
\end{enumerate}
\end{prop}

\subsection{Beyond genericity}
Let us remark on the case of {\it non}-generic cuspidal automorphic
representations $\pi$ of $L_k$. First of all, we know that such
representations really exist. This was proved in \cite{jiangsoudry}
by showing that $SO_{2l+1}$ has CAP-representations, i.e. cuspidal
automorphic representations $\tau$ which are nearly equivalent to a
subquotient of $\mathcal A_{Eis}(G)$. CAP representations are
expected to give counterexamples to the naively generalized
Ramanujan Conjecture, which says that for each cuspidal automorphic
representation $\tau$ all local components $\tau_p$ are tempered. On
the other hand, reinterpreting Shahidi's conjecture in
\cite{shahidi1} on the holomorphy of local $L$-functions associated
to tempered local representations, each tempered local L-packet
should contain a locally generic irreducible admissible
representation. Jiang and Soudry have proved this conjecture for
$SO_{2l+1}$ by virtue of \cite{jiangsoudry01} and
\cite{jiangsoudry02}.
%Kim has proved this conjecture for $SO_{2l+1}$ in \cite{kim05}.
So vaguely speaking in terms of L-packets, cuspidal
automorphic forms should disintegrate into the generic ones and the
other part containing the CAP representations - each world being
non-empty. Still, generic cuspidal representations should
``generate'' the whole cuspidal spectrum in the following way: It is
conjectured (cf. \cite{jiangsoudry} Conj. 1.1) that for each
cuspidal automorphic representation $\tau$ of $SO_{2l+1}(\A)$ there
is a (possibly not proper) parabolic subgroup $P'=L'N'$ and a
generic cuspidal automorphic representation $\tau'$ of $L'$ such
that $\tau$ is nearly equivalent to an irreducible constituent of
Ind$^{SO_{2l+1}(\A)}_{P'(\A)}[\tau']$. For $SO_{2l+1}$ the cuspidal
datum $(P',\tau')$ is an invariant for $\tau$ (up to near
equivalence), see \cite{jiangsoudry} Cor. 3.3.(3). In this setup,
generic cuspidal representations $\tau$ should be characterized by
$P'$ being the whole group $SO_{2l+1}$ and CAP representations by
$P'$ being proper.

\section{Residual Eisenstein Cohomology}

\subsection{Residual Eisenstein Classes for $k<n$}
We are ready to state the first of our two main Theorems.

\begin{thm}\label{mainthm1}
Let $0\neq[\beta]$ be a class of type $(\pi,w)$,
$\pi=\chi\wt\pi\in\varphi_{P_k}$ with $\wt\pi=\sigma\hat\otimes\tau$
a globally generic cuspidal automorphic representation of $L_k(\A)$
and $w\in W^{P_k}$ represented by a homomorphism the image of which
contains only functions $f\otimes 1\in\textrm{\emph{I}}_{P_k,\wt\pi}$ for
which $E_{P_k}(f,\Lambda)$ has a pole at the point
$\Lambda_w=-w(\lambda+\rho)|_{(\a_k)_\C}\in C$. Then:
\begin{enumerate}
\item $\pi=\chi(\sigma\hat\otimes\tau)$ is either of the form
\begin{enumerate}
\item $d\chi=\Lambda_w=\frac{k}{2}\alpha|_{(\a_k)_\C}$, $\sigma$ a self-dual cuspidal automorphic representation of $GL_k(\A)$ and such that $L(\frac{1}{2},\sigma\times\tau)\neq 0$. If in addition $k$ is even and $\omega_\sigma\equiv\triv$, then $k\geq    4$ and $\sigma$ is a self-dual Weak Langlands functorial lift of a globally generic cuspidal automorphic representation of $SO_k(\A)$. Or
\item $d\chi=\Lambda_w=k\alpha|_{(\a_k)_\C}$, $k$ even and $\sigma$ a self-dual cuspidal automorphic representation of $GL_k(\A)$ such that $L(s,\sigma\times\tau)$ has a pole at $s=1$.
\end{enumerate}
\item The degree $q'$ of the residual Eisenstein cohomology class constructible from $E^q_\pi([\beta])$ as in Theorem \ref{thm:poles} is necessarily in the following range
\begin{enumerate}
\item If $d\chi=\frac{k}{2}\alpha|_{(\a_k)_\C}$
$$\frac{1}{2}(n^2+n)-\lceil \frac{k}{2}\rceil\leq q'\leq \frac{1}{2}(n^2+n)-1.$$
\item If $d\chi=k\alpha|_{(\a_k)_\C}$
$$\frac{1}{2}(n^2+n)-k\leq q'\leq \frac{1}{2}(n^2+n)-1 \quad. $$
\end{enumerate}
\end{enumerate}
\end{thm}
\begin{proof}
Part (1) follows directly from our Propositions \ref{s=1/2} and
\ref{s=1}. For (2) we insert the formula from Proposition \ref{1/2}
for the length of $w=\wij\in W^{P_k}$ giving rise to
$t_{(I,J)}=\frac{k}{2}$ into the equation $q'=q-2l(\wij)+\dim
N_k(\R)$ and then use the bounds for $q-l(w)$ established in Theorem
\ref{boundsq}. We do the same for $w=\wij\in W^{P_k}$ giving rise to
$t_{(I,J)}=k$, using Proposition \ref{1} (\ref{ineq1}). This then
proves (2).
\end{proof}

\subsection{Residual Eisenstein Classes for $k=n$}
It remains to settle the case of the Siegel parabolic subgroup. We
prove

\begin{thm}\label{mainthm2}
Let $0\neq[\beta]$ be a class of type $(\pi,w)$,
$\pi=\chi\wt\pi\in\varphi_{P_n}$ with $\wt\pi=\sigma$ a cuspidal
automorphic representation of $L_n(\A)=GL_n(\A)$ and $w\in W^{P_n}$
represented by a homomorphism the image of which contains only
functions $f\otimes 1\in\textrm{\emph{I}}_{P_n,\wt\pi}$ for which
$E_{P_n}(f,\Lambda)$ has a pole at the point
$\Lambda_w=-w(\lambda+\rho)|_{(\a_n)_\C}\in C$. Then:
\begin{enumerate}
\item $\pi=\chi\sigma$ is of the form
\begin{enumerate}
\item[] $d\chi=\Lambda_w=\frac{n}{2}\alpha|_{(\a_n)_\C}$ and $\sigma$ is self-dual. If in addition $n$ is even and $\omega_\sigma\equiv\triv$, then $n\geq  4$ and $\sigma$ is a self-dual Weak Langlands functorial lift of a globally generic cuspidal automorphic representation of $SO_n(\A)$.
\end{enumerate}
\item The degree $q'$ of the residual Eisenstein cohomology class constructible from $E^q_\pi([\beta])$ as in Theorem \ref{thm:poles} is necessarily in the following range
$$\frac{1}{2}(n^2+n)-\lceil \frac{n}{2}\rceil\leq q'\leq \frac{1}{2}(n^2+n)-1.$$
\end{enumerate}
\end{thm}
\begin{proof}
As before, part (1) follows already from Proposition \ref{s=1n}. For
(2) we again insert the formula from Proposition \ref{1/2} for the
length of $w=\wij\in W^{P_n}$ giving rise to $t_{(I,J)}=\frac{n}{2}$
into the equation $q'=q-2l(\wij)+\dim N_n(\R)$ and then use the
bounds for $q-l(w)$ established in Theorem \ref{boundsq},
respectively in Proposition \ref{deggl}.
\end{proof}

\subsection{A remark on the lower bound}
Recall that by \cite{vozu}, Table 8.2, $n$ is the lowest possible
degree in which there could be non-trivial, square-integrable,
residual Eisenstein cohomology other than that coming from the
trivial representation. In fact, $n\leq \frac{1}{2}(n^2+n)-\lceil
\frac{k}{2}\rceil$ for all $n\geq 2$ and $k$, resp.
$n\leq\frac{1}{2}(n^2+n)-k$ for all $n\geq 3$ and even $k$. We do
not know if the lower bounds established by us in Theorems
\ref{mainthm1} and \ref{mainthm2} are in fact sharp.

\section{Regular Eisenstein Cohomology}\label{regec}

\subsection{}
We conclude the paper discussing regular Eisenstein cohomology
classes. Therefore, let $0\neq[\beta]$ be a class of type $(\pi,w)$,
$\pi=\chi\wt\pi\in\varphi_{P_k}$ with $\wt\pi=\sigma\hat\otimes\tau$
a globally generic cuspidal automorphic representation of $L_k(\A)$
and $w\in W^{P_k}$. \\\\Obviously, if either $d\chi$ is neither
$\frac{k}{2}\alpha|_{(\a_k)_\C}$ nor $k\alpha|_{(\a_k)_\C}$ or if
$\wt\pi$ is not of the form described in Theorem \ref{mainthm1} (1)
or \ref{mainthm2} (1), then for any $f\in\textrm{I}_{P_k,\wt\pi}$
the associated Eisenstein series $E_{P_k}(f,\Lambda)$ will be
holomorphic at $\Lambda=d\chi$. This is the content of our section
\ref{sec:poles}. In particular, in this case, the image of a
homomorphism $\beta$ representing the class $[\beta]$ can contain
only tensors $f\otimes\frac{d^{\nu}}{d\Lambda^{\nu}}$ for which the
associated Eisenstein series $E_{P_k}(f,\Lambda)$ is holomorphic at
$\Lambda=d\chi$. If in addition $\nu=0$, i.e.
$\frac{d^{\nu}}{d\Lambda^{\nu}}=1$, then
$E^q_\pi([\beta])$ is a non-trivial regular Eisenstein cohomology
class in degree $q$ with

$$\frac{1}{2}\left(\frac{k(k-1)}{2}+\lfloor\frac{k}{2}\rfloor+l^2+l\right)+l(w)\leq q\leq
\frac{1}{2}\left(\frac{(k-1)(k+4)}{2}-\lfloor\frac{k}{2}\rfloor+l^2+l\right)+l(w),$$
see Theorem \ref{thm:holeis} and Theorem \ref{boundsq}.

\subsection{}
Now let $d\chi$ and $\wt\pi$ be as in Theorem \ref{mainthm1} (1) or
\ref{mainthm2} (1), i.e., there are $f\in\textrm{I}_{P_k,\wt\pi}$
such that the associated Eisenstein series $E_{P_k}(f,\Lambda)$ has a
pole at $\Lambda=d\chi$. Still, if $f$ gives rise to a local
component $f_{s,p}\in\ker A(s,\wt\pi_p)$, where $s$ is according to
$d\chi$ either $\frac{1}{2}$ or $1$, then the zero
$A(s,\wt\pi_p)f_{s,p}=0$ will cancel the simple pole of the global
operator $M(s,\wt\pi)$ and again $E_{P_k}(f,\Lambda)$ will be
holomorphic at $\Lambda=d\chi$. Let $0\neq[\beta]$ be a class of
type $(\pi,w)$, represented by a homomorphism $\beta$ whose image in
$\textrm{I}_{P_k,\wt\pi}\otimes S_\chi(\a^*)$ consists of tensors
$f\otimes\frac{d^{\nu}}{d\Lambda^{\nu}}$ with $f$ as above. That is,
$f$ gives rise to a local component $f_{s,p}\in\ker A(s,\wt\pi_p)$,
$s=\frac{1}{2},1$. Then (cf. section \ref{eisintw})
$$E_{P_k,\pi}(f\otimes\frac{d^{\nu}}{d\Lambda^{\nu}})=E_{P_k}(f,d\chi)$$
if $\nu=1$. Assuming this, it follows now again from Theorem
\ref{thm:holeis} that $E^q_\pi([\beta])$ is a non-trivial regular
Eisenstein cohomology class. Its degree $q$ is bounded by

$$\frac{1}{2}\left(\frac{k(k-1)}{2}+\lfloor\frac{k}{2}\rfloor+l^2+l\right)+l(w)\leq q\leq
\frac{1}{2}\left(\frac{(k-1)(k+4)}{2}-\lfloor\frac{k}{2}\rfloor+l^2+l\right)+l(w),$$
see Theorem \ref{boundsq}.

\end{document}